\RequirePackage{fix-cm}
\documentclass[smallextended]{svjour3}       
\usepackage{graphicx}
\usepackage{color}
\usepackage{mathptmx}      
\journalname{Journal of Applied Mathemtics and Computing}
 
\begin{document}

\title{Parallel and sequential hybrid methods 
for a finite family\\ of  \textcolor[rgb]{1.0,0.0,0.0}{asymptotically quasi $\phi$-nonexpansive mappings}}

\titlerunning{Parallel and sequential hybrid methods}        

\author{Pham Ky Anh $^1$         \and
        Dang Van Hieu $^2$ 
}

\authorrunning{P.K. Anh and D.V. Hieu} 

\email{anhpk@vnu.edu.vn}
\institute{$^{1,  2}$ Department of Mathematics, 
Vietnam National University, Hanoi \at
              334 Nguyen Trai, Thanh Xuan, Hanoi,
Vietnam \\
              Tel.: +84-043-8581135\\
              Fax: +84-043-8588817\\
              \email{anhpk@vnu.edu.vn}           
           \and
             \email{dv.hieu83@gmail.com}
}

\date{Received: date / Accepted: date}

\maketitle

\begin{abstract}
In this paper we study some novel parallel and sequential hybrid methods for finding a common fixed point of a 
finite family of asymptotically quasi $\phi$-nonexpansive mappings. The results presented here modify and extend some previous 
results obtained by several authors.
\end{abstract}

\section{Introduction}\label{intro}
Let $C$ be a nonempty closed convex subset of a Banach space $E$. A mapping $T:C\to C$ is said to be nonexpansive if 
$$\left\|Tx-Ty\right\|\le \left\|x-y\right\|,   \forall x,y\in C.$$
In 2005, Matsushita and Takahashi \cite{MT2005} proposed the following hybrid method, combining Mann iterations with projections
onto closed convex subsets, for finding  a fixed point of a relatively nonexpansive mapping $T$: 
\begin{eqnarray*}
&&x_0\in C,\\ 
&&y_n=J^{-1}(\alpha_n Jx_n+(1-\alpha_n)JTx_n),\\
&&C_n=\left\{v\in C:\phi(v,y_n)\le \phi(v,x_n)\right\},\\
&&Q_n=\left\{v\in C : \left\langle Jx_0-Jx_n,x_n-v\right\rangle \ge 0\right\},\\
&&x_{n+1}=\Pi_{C_n\bigcap Q_n}x_0,n\ge 0.
\end{eqnarray*}
This algorithm has been modified and generalized for finding a common fixed point of a finite or infinite family of relatively nonexpansive 
mappings by several authors, such as Takahashi et al. \cite{TTK2008}, Takahashi and Zembayashi \cite{TZ2008}, Wang and Xuan \cite{WX2013}, Reich and Sabach 
\cite{RS2010a,RS2010b}, Kang, Su, and Zhang \cite{KSZ2010}, Plubtieng and Ungchittrakool \cite{PU2010}, etc... \\
In 2011, Liu \cite{L2011} introduced the following cyclic 
method for a finite family of relatively nonexpansive mappings:
\begin{eqnarray*}
&&x_0\in C,\\ 
& &y_n=J^{-1}(\alpha_n Jx_0+(1-\alpha_n)JT_{n ({\rm mod}) N}x_n),\\
&&C_n=\left\{v\in C:\phi(v,y_n)\le \alpha_n\phi(v,x_0)+(1-\alpha_n)\phi(v,x_n)\right\},\\
&&Q_n=\left\{v\in C : \left\langle Jx_0-Jx_n,x_n-v\right\rangle \ge 0\right\},\\
&&x_{n+1}=\Pi_{C_n\bigcap Q_n}x_0,n\ge 0.
\end{eqnarray*}
Very recently, Anh and Chung \cite{AC2013} considered the following parallel method for a finite family of 
relatively nonexpansive mappings:
\begin{eqnarray*}
&&x_0\in C,\\ 
&&y_n^i=J^{-1}(\alpha_n Jx_n+(1-\alpha_n)JT_{i}x_n), \quad i=1,\ldots,N,\\
&&i_n=\arg\max_{1\le i\le N}\left\{\left\|y_n^i-x_n \right\|\right\}, \quad  \bar{y}_n := y_n^{i_n},\\
&&C_n=\left\{v\in C:\phi(v,\bar{y}_n)\le \phi(v,x_n)\right\},\\
&&Q_n=\left\{v\in C : \left\langle Jx_0-Jx_n,x_n-v\right\rangle \ge 0\right\},\\
&&x_{n+1}=\Pi_{C_n\bigcap Q_n}x_0,n\ge 0.
\end{eqnarray*}

According to this algorithm, the intermediate approximations $y_n^i$ can be found in parallel. Then among all 
$y_n^i,  i=1,\ldots,N,$  the farest element from $x_n$, denoted by $\bar{y}_n$, is chosen. After that, two convex closed subsets $C_n$ 
and $Q_n$ containing the set of common fixed points are constructed. The next approximation $x_{n+1}$ is defined as the generalized projection of $x_0$ onto the intersection $C_n\bigcap Q_n$.  \\
Further, some generalized hybrid projection methods have been introduced for families of hemi-relatively or weak relatively 
nonexpansive mappings (see, \cite{KSZ2010,SWX2009,WK2013}).\\
On the other hand, there has been an increasing interest in the class of \textcolor[rgb]{1.0,0.0,0.0}{asymptotically quasi $\phi$-nonexpansive mappings}
 (c.f., \cite{CKW2010,CY2013,D2013a,D2013b,DB2013,HLK2004,KK2012,K2011,LL2013,TCLL2010,ZCK2013}), 
which is a generalization of the class of quasi $\phi$- nonexpansive mappings. The last one contains the class of relatively 
nonexpansive mappings as a proper subclass.\\ 
Unfortunately, many hybrid algorithms for (relatively) nonexpansive mappings cannot be directly extended to \textcolor[rgb]{1.0,0.0,0.0}{asymptotically quasi $\phi$-nonexpansive mappings}.\\ 
The aim of this paper is to combine a parallel splitting-up technique proposed in \cite{AC2013} with a monotone hybrid iteration method
(see, \cite{SLZ2011}) for finding a common fixed point of a finite family of \textcolor[rgb]{1.0,0.0,0.0}{asymptotically quasi $\phi$ -nonexpansive} mappings. 
The organization of the paper is as follows: In Section 2 we collect some 
definitions and results which are used in this paper. Section 3 deals with the convergence analysis of the proposed parallel and 
sequential hybrid algorithms. \textcolor[rgb]{1.0,0.0,0.0}{Finally, a numerical example shows that even in the sequential mode, our parallel hybrid method is faster than the corresponding sequential one [20].}

\section{Preliminaries}\label{sec:1}
In this section we recall some definitions and results needed for further investigation. We refer the interested reader to \cite{AR2006,D1975} for more details. 

\begin{definition}
A Banach space $X$ is called
\begin{itemize}
\item  [$1)$] strictly convex if the unit sphere $S_1(0) = \{x \in X: ||x||=1\}$ is strictly convex, i.e., the inequality $||x+y|| <2$ holds for 
all $x, y \in S_1(0),  x \ne y  ;$ 
 \item [$2)$] uniformly convex if for any given $\epsilon >0$ there exists $\delta = \delta(\epsilon)  >0$ such that for all $x,y \in X$ 
with $\left\|x\right\| \le 1, \left\|y\right\|\le 1,\left\|x-y\right\| = \epsilon$ the inequality $\left\|x+y\right\|\le 2(1-\delta)$ holds;
\item [$3)$] smooth if the limit
\begin{equation}\label{smoonthspa}
\lim\limits_{t \to 0} \frac{\left\|x+ty\right\|-\left\|x\right\|}{t}
\end{equation}
exists for all $x,y \in S_1(0)$;
\item[$4)$] uniformly smooth if the limit $(\ref{smoonthspa})$ exists uniformly for all $x,y\in S_1(0)$.
\end{itemize}
\end{definition}
Let $E$ be a real Banach space with the dual $E^*$ and $J:E\to 2^{E^*}$ is the normalized duality mapping defined by
$$ J(x)=\left\{f\in E^*: \left\langle f,x \right\rangle=\left\|x\right\|^2=\left\|f\right\|^2\right\}.$$
The following basic properties of the geometry of $E$ and its normalized duality mapping $J$ 
can be found in \cite{C1990}:
\begin{itemize}
\item [i)] If $E$ is a reflexive and strictly convex Banach space, then $J^{-1}$ is norm to weak $^*$  continuous;
\item [ii)] If $E$ is a smooth, strictly convex, and reflexive Banach space, then the normalized duality mapping $J:E\to 2^{E^*}$ is 
single-valued, one-to-one, and onto;
\item [iii)] If $E$ is a uniformly smooth Banach space, then $J$ is uniformly continuous on each bounded subset of $E$;
\item [iv)] A Banach space $E$ is uniformly smooth if and only if $E^*$ is uniformly convex;
\item [v)] Each uniformly convex Banach space $E$ has the Kadec-Klee property, i.e., for any
sequence $\left\{x_n\right\} \subset E$, if $x_n \rightharpoonup x \in E$ and $\left\| x_n \right\|\to \left\|x\right\|$, then $x_n \to x$.
\end{itemize}
Next we assume that $C$ is a nonempty closed convex subset of a smooth, strictly convex, and reflexive Banach space $E$. 
Consider the Lyapunov functional $\phi: E\times E \to R_+$ defined by
$$ \phi(x,y)=\left\|x\right\|^2-2\left\langle x,Jy\right\rangle+\left\|y\right\|^2, \forall x, y \in E. $$
From the definition of $\phi$, we have
\begin{equation}\label{eq:proPhi}
\left(\left\|x\right\|-\left\|y\right\|\right)^2\le \phi(x,y)\le \left(\left\|x\right\|+\left\|y\right\|\right)^2.
\end{equation}
The generalized projection $\Pi_C:E\to C$ is defined by 
\[ \Pi_C(x)=\arg\min_{y\in C}\phi(x,y). \]
\begin{lemma} \label{lem.proProjec} \cite{A1996}
Let E be a smooth, strictly convex, and reflexive Banach space and $C$ be a nonempty closed convex subset of $E$. Then the following 
conclusions hold:
\begin{itemize}
\item [i)] $\phi (x,\Pi_C (y))+\phi (\Pi_C (y),y) \le \phi (x,y), \forall x \in C, y \in E$;
\item [ii)] if $x\in E, z \in C$, then $z= \Pi_C(x)$ iff $\left\langle z-y, Jx-Jz\right\rangle \ge 0, \forall y \in C$;
\item [iii)] $\phi (x,y)=0$ iff $x=y$.
\end{itemize}
\end{lemma}

\begin{lemma}\label{lem.prophi}\cite{A1996}
Let $E$ be a uniformly convex and uniformly smooth real Banach space, $\left\{x_n\right\}$ and $\left\{y_n\right\}$ be two sequences 
in $E$. If $\phi(x_n,y_n)\to 0$ and either $\left\{x_n\right\}$ or $\left\{y_n\right\}$ is bounded, then $\left\|x_n-y_n\right\|\to 0$ as 
$n\to \infty$.
\end{lemma}
Let $C$ be a nonempty closed convex subset of  a smooth, strictly convex, and reflexive Banach space $E$,  $T:C\to C$ be a 
mapping, and $F(T)$ be the set of fixed points of $T$. A point $p\in C$ is said to be an asymptotic fixed point of $T$ if there exists a sequence 
$\left\{x_n\right\} \subset C$ such that $x_n \rightharpoonup p$ and $\left\|x_n-Tx_n \right\|\to 0$ as 
$n\to +\infty$. The set of all asymptotic fixed points of $T$ will be denoted by 	 $\tilde{F}(T)$.
\begin{definition} \label{def.RNM}
A mapping $T:C\to C$ is called
\begin{itemize}
\item [i)] relatively nonexpansive mapping if $F(T)\ne \O, F(T)=\tilde{F}(T)$, and $$ \phi(p,Tx)\le \phi (p,x), \forall p\in F(T),  \forall x\in C
;$$
\item [ii)] closed if for any sequence $\left\{x_n\right\} \subset C, x_n \to x$ and $Tx_n \to y$, then $Tx=y$;
\item [iii)] quasi $\phi$ - nonexpansive mapping (or hemi-relatively nonexpansive mapping) if $F(T) \ne \O$ and $$ \phi(p,Tx)\le 
\phi (p,x), \forall p\in F(T), \forall x\in C;$$
\item [iv)] \textcolor[rgb]{1.0,0.0,0.0}{asymptotically quasi $\phi$-nonexpansive} if $F(T)\ne \O$ and there exists a sequence 
$\left\{k_n\right\} \subset [1,+\infty)$ with $k_n \to 1$ as $n\to +\infty$ such that $$\phi(p,T^n x)\le k_n\phi (p,x), \forall n\ge 1, \forall p\in 
F(T), \forall x\in C; $$
\item [v)] uniformly $L$ - Lipschitz continuous, if there exists a constant $L>0$ such that
\[\left\|T^n x-T^n y \right\|\le L\left\|x- y \right\|, \forall n \geq 1 ,  \forall x, y \in C.\]
\end{itemize}
\end{definition}
\begin{lemma} \cite{CKW2010} \label{lem.proCloseConvex}
Let $E$ be a real uniformly smooth and strictly convex Banach space with 
Kadec-Klee property, and $C$ be a nonempty closed convex subset of $E$. Let $T:C\to C$ be a closed and \textcolor[rgb]{1.0,0.0,0.0}{asymptotically quasi $\phi$-nonexpansive} mapping with a sequence $\left\{k_n\right\}\subset [1,+\infty), k_n\to 1$. Then $F(T)$ is a closed convex subset of $C$.
\end{lemma}
\begin{lemma} \cite{CKW2010,KT2009,MT2005} \label{lem.MNOper}
Let $E$ be a strictly convex reflexive smooth Banach space, $A$ be a maximal monotone operator of $E$ into $E^*$, and 
$J_r=(J+rA)^{-1}J:E\to D(A)$ be the resolvent of $A$ with $r>0$. Then, 
\begin{itemize}
\item [$i)$] $ F(J_r)=A^{-1}0$;
\item [$ii)$] $\phi(u,J_r x)\le \phi(u,x)$ for all $u\in A^{-1}0$ and $x\in E$.
\end{itemize}
\end{lemma}
\begin{lemma}\label{Closed-of-Jr}\cite{SLZ2011}
Let $E$ be a uniformly convex and uniformly smooth Banach space, $A$ be a maximal monotone operator from $E$ to $E^*$, and 
$J_r$ be a resolvent of $A$. Then $J_r$ is closed hemi-relatively nonexpansive mapping.
\end{lemma}
\section{Main results}\label{sec:2}
\subsection{\bf Parallel hybrid methods}\label{subsec:2.1}
Assume that $T_i, i=1,2,\ldots, N$,  are \textcolor[rgb]{1.0,0.0,0.0}{asymptotically quasi $\phi$-nonexpansive} mappings with a sequence
$\left\{k_n^i\right\} \subset [1,+\infty), k_n^i\to 1$, i.e.,  
$F(T_i) \ne \O$, and 
$$\phi(p,T_i^n x)\le k_n^i\phi (p,x), \forall n\ge 1, \forall p\in 
F(T_i), \forall x\in C. $$ 
Throughout this paper we suppose that the set $F=\bigcap^N_{i=1} F(T_i) $ is nonempty.\\ 
Then, putting $k_n := \max\{ k_n^i: i=1,\ldots,N \} 
$, we have $k_n\subset [1,+\infty), k_n\to 1, $ and 
$$\phi(p,T^n_i x)\le k_n\phi (p,x),  \forall i =1,\ldots, N,   \forall n\ge 1, \forall p\in F, \forall x\in C. $$
\textcolor[rgb]{1.0,0.0,0.0}{In the following theorems we will assume that the set $F=\bigcap^N_{i=1} F(T_i) $ 
is nonempty and bounded in $C$, i.e., there exists a positive number $\omega$ such that 
$F \subset \Omega := \{u \in C: ||u|| \leq \omega \}$.} 
\begin{theorem}\label{theo.converges}
Let $E$ be a real uniformly smooth and uniformly convex Banach space and $C$ be a nonempty closed convex subset of $E$. Let 
$\left\{T_i\right\}_{i=1}^N:C\to C$ be a finite family of \textcolor[rgb]{1.0,0.0,0.0}{asymptotically quasi $\phi$-nonexpansive} mappings with a 
sequence $\left\{k_n\right\} \subset [1,+\infty), k_n\to 1$. Moreover, suppose for each $i\ge 1$, the 
mapping $T_i$ is uniformly $L_i$ - Lipschitz continuous and the set $F=\bigcap^N_{i=1} F(T_i) $ 
is nonempty and bounded in $C$. Let $\left\{x_n\right\}$ be the sequence generated by
\begin{eqnarray*}
&&x_0\in C, C_0 := C,\\
&&y_n^i=J^{-1}\left(\alpha_n Jx_n +(1-\alpha_n)JT^n_i x_n\right),i=1,2,\ldots,N,\\
&&i_n=\arg\max_{1\le i\le N}\left\{\left\|y_n^i-x_n \right\|\right\}, \quad   \bar{y}_n := y_n^{i_n
},\\
&&C_{n+1}:=\left\{v\in C_n:\phi(v,\bar{y}_n)\le \phi(v,x_n)+\epsilon_n\right\},\\
&&x_{n+1}=\Pi_{C_{n+1}}x_0, n\ge 0,
\end{eqnarray*}
where \textcolor[rgb]{1.0,0.0,0.0}{$\epsilon_n:= (k_n-1)(\omega + ||x_n||)^2$,} and $\left\{\alpha_n\right\}$ is a sequence in $[0,1]$ such that 
$\lim\limits_{n\to\infty}\alpha_n =0$.
Then $\left\{x_n\right\}$ converges strongly to $x^{\dagger}:=\Pi_F x_0$.
\end{theorem}
\begin{proof}
The proof of Theorem $\ref{theo.converges}$ is divided into five steps.\\
\textit{Step 1.} Claim that $F$ and $C_n$ are closed and convex subsets of $C$.\\
Indeed, from the uniform $L_i$ - Lipschitz continuty of $T_i$, $T_i$ is $L_i$ - Lipschitz continuty. Hence $T_i$ is continuous. This implies that $T_i$ is closed. By Lemma $\ref{lem.proCloseConvex}$, 
$F(T_i)$ is closed and convex subset of $C$ for all $i=1,2,\ldots, N$. Hence, $F=\bigcap^N_{i=1} F(T_i)$ is closed and convex.
Further, $C_0=C$ is closed and convex by the assumption. Suppose that $C_n$ is a closed and convex subset of $C$ for some 
$n\ge 0$.  From the inequality $\phi(v,\bar{y}_n)\le \phi(v,x_n)+\epsilon_n$, we obtain
$$\left\langle v,Jx_n -J\bar{y}_n\right\rangle \le \frac{1}{2}\left(\left\|x_n \right\|^2-\left\|\bar{y}_n \right\|^2+\epsilon_n\right).$$
Therefore,
$$ C_{n+1}=\left\{v\in C_n:\left\langle v,Jx_n -J\bar{y}_n\right\rangle \le \frac{1}{2}\left(\left\|x_n \right\|^2-\left\|\bar{y}_n \right\|^
2+\epsilon_n\right)\right\}, $$
which implies that $C_{n+1}$ is closed and convex. Thus, $C_n$ is closed and convex subset of $C$ for all $n\ge 0$, and $\Pi_C x_0$ 
and $\Pi_{C_n} x_0$ are well-defined.\\
\textit{Step 2.} Claim that $F\subset C_n$ for all $n\ge 0$. \\
Observe first that $F\subset C_0=C$. Now suppose 
$F\subset C_n$ for some $n\ge 0$. For each $u\in F$, by the convexity of $\left\|.\right\|^2$, we have
\begin{eqnarray*}
\phi(u,\bar{y}_n)&=&\left\|u\right\|^2-2\left\langle u,J\bar{y}_n\right\rangle+\left\|\bar{y}_n \right\|^2\\ 
&=&\left\|u\right\|^2-2\alpha_n\left\langle u,Jx_n\right\rangle -2(1-\alpha_n)\left\langle u,JT_{i_n}^n x_n\right\rangle\\
&&+\left\|\alpha_n Jx_n +(1-\alpha_n)JT^n_{i_n} x_n \right\|^2\\
&\le &\left\|u\right\|^2-2\alpha_n\left\langle u,Jx_n\right\rangle -2(1-\alpha_n)\left\langle u,JT_{i_n}^n x_n\right\rangle\\
&&+\alpha_n\left\|x_n\right\|^2+(1-\alpha_n)\left\|T^n_{i_n} x_n\right\|^2\\
&=&\alpha_n\phi(u,x_n)+(1-\alpha_n)\phi(u,T^n_{i_n} x_n)\\
&\le& \alpha_n\phi(u,x_n)+k_n(1-\alpha_n)\phi(u,x_n)\\
&\le& \phi(u,x_n)+(k_n-1)(1-\alpha_n)\phi(u,x_n) \\
&\le&\phi(u,x_n)+(k_n-1)(\omega+||x_n||)^2 \\
&=& \phi(u,x_n)+\epsilon_n.
\end{eqnarray*}
This implies that $u\in C_{n+1}$. Hence $F\subset C_{n+1}$. By induction, we obtain $F\subset C_n$ for all $n\ge 0$. For each 
$u\in F\subset C_n$, by $x_n=\Pi_{C_n}x_0$ and 
Lemma $\ref{lem.proProjec}$, we have
$$ \phi(x_n,x_0)\le \phi(u,x_0)-\phi(u,x_n)\le\phi(u,x_0).$$
Therefore, the sequence $\left\{\phi(x_n,x_0)\right\}$ is bounded. The boundedness of the sequence $\left\{x_n\right\}$ is followed 
from 
relation $(\ref{eq:proPhi})$. \\
\textit{Step 3.} Claim that the sequence $\left\{x_n\right\}$ converges strongly to some point $p\in C$ as $n \to 
\infty$.\\
 By the construction of $C_n$, we have $C_{n+1}\subset C_n$ and $x_{n+1}=\Pi_{C_{n+1}}x_0 \in C_{n+1}$
Now taking into account  $x_n=\Pi_{C_n}x_0, \,x_{n+1}\in C_n$ and using Lemma $\ref{lem.proProjec}$, we get
$$\phi(x_n,x_0)\le \phi(x_{n+1},x_0)-\phi(x_{n+1},x_n)\le\phi(x_{n+1},x_0).$$
This implies that $\left\{\phi(x_n,x_0)\right\}$ is nondecreasing. Therefore, the limit of $\left\{\phi(x_n,x_0)\right\}$ exists. We also have $x_{m}\in C_m\subset C_n$ for all $m\ge n$. From Lemma $\ref{lem.proProjec}$ and $x_n=\Pi_{C_n}x_0$, we obtain
$$\phi(x_m,x_n)\le\phi(x_m,x_0)-\phi(x_n,x_0)\to 0,$$
as $m,n\to\infty$. This together with Lemma $\ref{lem.prophi}$ implies that $||x_m-x_n||\to 0$. Hence, $\left\{x_n\right\}$ is Cauchy sequence. Since $E$ is complete and $C$ is closed, we get
\begin{equation}\label{eq:4}
\lim\limits_{n\to\infty}x_n=p\in C.
\end{equation}
\textit{Step 4.} Claim that $p\in F$. \\
Indeed, observing that  
\begin{equation}\label{eq:5}
\phi(x_{n+1},x_n)\le \phi(x_{n+1},x_0)-\phi(x_{n},x_0) \to 0,
\end{equation}
and
\begin{equation}\label{eq:5**}
\left\|x_{n+1}-x_n\right\|\to 0.
\end{equation}
In view of $x_{n+1} \in C_{n+1}$ and by the construction of $C_{n+1}$, we obtain
\begin{equation}\label{eq:6}
\phi(x_{n+1},\bar{y}_n) \le \phi(x_{n+1},x_n) +\epsilon_n.
\end{equation}
Recalling that the set $F$ and the sequence $\left\{x_n\right\}$ are bounded, and putting\\$M=\sup\left\{\left\|x_n\right\|:n=1,2,\ldots\right\}$
, we get
\textcolor[rgb]{1.0,0.0,0.0}{\begin{equation}\label{eq:7}
\epsilon_n=(k_n-1)\left(\omega+||x_n||\right)^2\le (k_n-1)\left(\omega+M\right)^2 \to 0.
\end{equation}}
From $(\ref{eq:5}),(\ref{eq:6}),(\ref{eq:7})$, we obtain $\phi(x_{n+1},\bar{y}_n) \to 0$ as $n\to \infty$. This together with 
Lemma $\ref{lem.prophi}$ implies that $\left\|x_{n+1}-\bar{y}_n\right\| \to 0$. Therefore, from $(\ref{eq:5**})$, $||x_n - \bar{y}_n|| 
\to 0$. Further, by the definition of $i_n$, we have 
$\left\|x_n-y_n^{i}\right\| \leq ||x_n - \bar{y}_n|| \to 0$ as $n\to \infty$ for all $i=1,2,\ldots,N$, hence, from $(\ref{eq:4})$ we obtain 
\begin{equation}\label{eq:9}
\lim\limits_{n\to \infty}y_n^i=p, \quad i=1,2,\ldots,N.
\end{equation}
From the relation $y_n^i=J^{-1}\left(\alpha_n Jx_n +(1-\alpha_n)JT^n_i x_n\right)$ we obtain
\begin{equation}\label{eq:10}
\left\|Jy_n^i-JT^n_i x_n\right\|=\alpha_n\left\|Jx_n-JT^n_i x_n\right\|.
\end{equation}
Observing that $\left\{x_n\right\}$ is bounded, $T_i$ is uniformly $L_i$ - Lipschitz continuous, and the 
solution set $F$ is not empty, we have $||Jx_n - JT_i^nx_n|| \leq ||Jx_n||+||JT_i^nx_n|| = ||x_n|| + ||T_i^n x_n|| \leq ||x_n|| +||T_i^nx_n 
- T_i^n\xi|| + ||\xi|| \leq ||x_n|| + L_i||x_n - \xi|| +|\xi||$, where $\xi \in F$ is an arbitrary fixed element. The last inequality proves the 
boundedness of the sequence  $\left\{\left\|Jx_n-JT^n_i x_n\right\|\right\}$ . Using $\lim\limits_{n\to \infty}\alpha_n=0$, from $(\ref 
{eq:10})$, we find
$$\lim\limits_{n\to \infty}\left\|Jy_n^i-JT^n_i x_n\right\|=0.$$
Since $J^{-1}:E^*\to E$ is uniformly continuous on each bounded subset of $E^*$, the last relation 
implies $\lim\limits_{n\to \infty}\left\|y_n^i-T^n_i x_n\right\|=0$. Hence, from $(\ref{eq:9})$ we obtain
\begin{equation}\label{eq:12}
\lim\limits_{n\to \infty}T^n_i x_n=p, \quad i =1,\ldots, N.
\end{equation}
By $(\ref{eq:4}),(\ref{eq:12})$ and the uniform $L_i$- Lipschitz continuity of $T_i$, we have
\begin{eqnarray*}
\left\|T^{n+1}_i x_n-T^{n}_i x_n\right\|\le& \left\|T^{n+1}_i x_n-T^{n+1}_i x_{n+1}\right\| +\left\|T^{n+1}_i x_{n+1} -x_{n+1}
\right\|\\
&+\left\|x_{n+1}-x_n\right\|+\left\|x_n-T^{n}_i x_n\right\|\\ 
 \le& \left(L_i +1\right) \left\|x_{n+1}-x_n\right\| +\left\|T^{n+1}_i x_{n+1} -x_{n+1}\right\|\\
&+\left\|x_n-T^{n}_i x_n\right\| \to 0.
\end{eqnarray*}
Hence, $\lim\limits_{n\to \infty}T^{n+1}_i x_n=p$, i.e., $T^{n+1}_i x_n=T_iT^{n}_i x_n\to p$ as $n\to \infty$. In view of the 
continuty of $T_i$ and $(\ref{eq:12})$, it follows that $T_i p=p$ for all $i=1,2,\ldots,N$. Therefore $p\in F.$\\
\textit{Step 5.} Claim that $p=x^{\dagger}:=\Pi_F(x_0)$. \\
Indeed, since $x^{\dagger}=\Pi_F(x_0)\in F \subset 
C_n$ and $x_n=\Pi_{C_n}(x_0)$, from Lemma $\ref{lem.proProjec}$, we have
\begin{equation}\label{eq:13}
\phi(x_n,x_0)\le \phi(x^\dagger,x_0)-\phi(x^\dagger,x_n)\le \phi(x^\dagger,x_0).
\end{equation}
Therefore, 
\begin{eqnarray*}
\phi(x^\dagger,x_0)&\ge& \lim\limits_{n\to \infty}\phi(x_n,x_0)=\lim\limits_{n\to \infty}\left\{\left\|x_n\right\|^2-2\left\langle x_n,Jx_0
\right\rangle+\left\|x_0\right\|^2\right\}\\ 
&=& \left\|p\right\|^2-2\left\langle p,Jx_0\right\rangle+\left\|x_0\right\|^2\\
&=&\phi(p,x_0).
\end{eqnarray*}
From the definition of $x^\dagger$, it follows that $p=x^\dagger$. The proof of Theorem $\ref{theo.converges}$ is complete.
\end{proof}
\begin{remark}
If  in Theorem $\ref{theo.converges}$ instead of the uniform Lipschitz continuty of the operators $T_i, ~ i=1,\ldots, N$, we require  their closedness and asymptotical regularity
 \cite{CQK2009}, i.e., for any bounded subset $K$ of $C$,
$$\lim_{n\to\infty}\sup\left\{\left\|T^{n+1}_ix-T^n_i x\right\|:x\in K \right\}, i=1,\ldots,N, $$
then we  obtain the strong convergence of a simplier method  than the corresponding ones in Cho, Qin, and Kang \cite{CQK2009} and Chang, Kim, and Wang \cite{CKW2010}.
\end{remark} 
For the case $N=1,$  Theorem $\ref{theo.converges}$ gives the following monotone hybrid method, which modifies the corresponding algorithms in Kim 
and Xu \cite{KX2006}, as well as Kim and Takahashi (Theorems 3.1,  3.7, 4.1 \cite{KT2010}).
\begin{corollary}
Let $E$ be a real uniformly smooth and uniformly convex Banach space and $C$ be a nonempty closed convex subset of $E$. Let 
$T:C\to C$ be an \textcolor[rgb]{1.0,0.0,0.0}{asymptotically quasi $\phi$-nonexpansive} mapping with a 
sequence $\left\{k_n\right\} \subset [1,+\infty), k_n\to 1$. Moreover, suppose that the 
mapping $T$ is uniformly $L$ - Lipschitz continuous and the set $F(T) $ 
is nonempty and bounded in $C$. Let $\left\{x_n\right\}$ be the sequence generated by
\begin{eqnarray*}
&&x_0\in C, C_0 := C,\\
&&y_n=J^{-1}\left(\alpha_n Jx_n +(1-\alpha_n)JT^n x_n\right),\\
&&C_{n+1}:=\left\{v\in C_n:\phi(v,{y}_n)\le \phi(v,x_n)+\epsilon_n\right\},\\
&&x_{n+1}=\Pi_{C_{n+1}}x_0, n\ge 0,
\end{eqnarray*}
where \textcolor[rgb]{1.0,0.0,0.0}{$\epsilon_n=(k_n-1)(\omega+||x_n||)^2$} and $\left\{\alpha_n\right\}$ is a sequence in $[0,1]$ such that 
$\lim\limits_{n\to\infty}\alpha_n =0$.
Then $\left\{x_n\right\}$ converges strongly to $x^{\dagger}:=\Pi_{F(T)} x_0$.
\end{corollary}
Next, we consider a modified version of the algorithm proposed in Theorem $\ref{theo.converges}$.
\begin{theorem}\label{theo.extend1}
Let $E$ be a real uniformly smooth and uniformly convex Banach space and $C$ be a nonempty closed convex subset of $E$. Let 
$\left\{T_i\right\}_{i=1}^N:C\to C$ be a finite family of \textcolor[rgb]{1.0,0.0,0.0}{asymptotically quasi $\phi$-nonexpansive} mappings with a 
sequence $\left\{k_n\right\} \subset [1,+\infty), k_n\to 1$. Moreover, suppose for each $i\ge 1$, the 
mapping $T_i$ is uniformly $L_i$ - Lipschitz continuous and the set $F=\bigcap^N_{i=1} F(T_i) $ 
is nonempty and bounded in $C$. Let $\left\{x_n\right\}$ be the sequence generated by
\begin{eqnarray*}
&&x_0\in C, C_0 := C,\\
&&y_n^i=J^{-1}\left(\alpha_n Jx_0 +(1-\alpha_n)JT^n_i x_n\right),i=1,2,\ldots,N,\\
&&i_n=\arg\max_{1\le i\le N}\left\{\left\|y_n^i-x_n \right\|\right\}, \quad   \bar{y}_n := y_n^{i_n
},\\
&&C_{n+1}:=\left\{v\in C_n:\phi(v,\bar{y}_n)\le \alpha_n\phi(v,x_0)+(1-\alpha_n)\phi(v,x_n)+\epsilon_n\right\},\\
&&x_{n+1}=\Pi_{C_{n+1}}x_0, n\ge 0,
\end{eqnarray*}
where \textcolor[rgb]{1.0,0.0,0.0}{$\epsilon_n=(k_n-1)(\omega+||x_n||)^2$} and $\left\{\alpha_n\right\}$ is a sequence in $[0,1]$ such that 
$\lim\limits_{n\to\infty}\alpha_n =0$.
Then $\left\{x_n\right\}$ converges strongly to $x^{\dagger}:=\Pi_F x_0$.
\end{theorem}
\begin{proof} Following five steps in the proof  of Theorem $\ref{theo.converges}$, we can show that: \\
\noindent (i)   $C_n$ and $F$ are closed and convex subset of $C$ for all $n\ge 0$. Therefore, $\Pi_{C_n} x_0 , n\ge 0$ and $\Pi_F x_0$ are
well-defined.\\
\noindent (ii) $F\subset C_n$ for all $n\ge 0$. \\
Suppose  $F\subset C_n$ for some $n\ge 0 ~ (F\subset C_0=C)$. For each $u\in F$, using the convexity of $\left\|.\right\|^2$, we get
\begin{eqnarray*}
\phi(u,\bar{y}_n)&=&\left\|u\right\|^2-2\left\langle u,J\bar{y}_n\right\rangle+\left\|\bar{y}_n \right\|^2\\ 
&\le& \alpha_n\phi(u,x_0)+k_n(1-\alpha_n)\phi(u,x_n)\\
&\le &\alpha_n\phi(u,x_0)+(1-\alpha_n)\phi(u,x_n)+(k_n-1)(1-\alpha_n)\phi(u,x_n)\\
&\le &\alpha_n\phi(u,x_0)+(1-\alpha_n)\phi(u,x_n)+(k_n-1)(\omega+||x_n||^2)\\
&=& \alpha_n\phi(u,x_0)+(1-\alpha_n)\phi(u,x_n)+\epsilon_n.
\end{eqnarray*}
This implies that $u\in C_{n+1}$. Hence $F\subset C_{n+1}$. By induction, we obtain $F\subset C_n$ for all $n\ge 0$. \\
\noindent (iii) The sequence $\left\{x_n\right\}$ converges strongly to some point $p\in C$ as $n \to \infty$.\\
For each $u\in F\subset C_n$, using Lemma $\ref{lem.proProjec}$ and taking into account  that $x_n=\Pi_{C_n}x_0$, we have
$$ \phi(x_n,x_0)\le \phi(u,x_0)-\phi(u,x_n)\le\phi(u,x_0).$$
Therefore, the sequence $\left\{\phi(x_n,x_0)\right\}$ is bounded. From $\left(\ref{eq:proPhi}\right)$, $\left\{x_n\right\}$ is also bounded. Since $C_{n+1}\subset C_n$ and
 $x_{n+1}=\Pi_{C_n+1}x_0 \in C_n$ for all $n\ge 0$, by  Lemma $\ref{lem.proProjec}$ we have
$$\phi(x_n,x_0)\le\phi(x_{n+1},x_0)-\phi(x_{n+1},x_n)\le \phi(x_{n+1},x_0).$$
Thus, the sequence $\left\{\phi(x_n,x_0)\right\}$ is nondecreasing, hence it has a finite limit as $n \to \infty$. Moreover, for all $m\ge n$, we also have $x_m=\Pi_{C_m}x_0\in C_m\subset C_n.$  From 
$x_{n}=\Pi_{C_n}x_0$ and Lemma $\ref{lem.proProjec}$, we obtain
\begin{equation}\label{eq:5*}
\phi(x_{m},x_n)\le\phi(x_{m},x_0)-\phi(x_n,x_0)\to 0
\end{equation}
as $m,n\to \infty$. Lemma $\ref{lem.prophi}$ yields $\left\|x_m-x_n\right\|\to 0$ as $m,n\to \infty$. Therefore, $\left\{x_n\right\}$ is a 
Cauchy sequence in $C$. Since $E$ is Banach space and $C$ is closed, $x_n\to p\in C$ as $n\to\infty$.\\
\noindent (iv) $p\in F$.\\
 In view of $x_{n+1} \in C_{n+1}$ and by the construction of $C_{n+1}$, we get
\begin{equation}\label{eq:6*}
\phi(x_{n+1},\bar{y}_n) \le \alpha_n\phi(x_{n+1},x_0)+(1-\alpha_n)\phi(x_{n+1},x_n) +\epsilon_n.
\end{equation}
Using $\lim_{n\to\infty}\alpha_n =0$, relations $(\ref{eq:5*}),(\ref{eq:6*})$, and noting that $\epsilon_n\to 0$, we find 
$\phi(x_{n+1},\bar{y}_n) \to 0$ as $n\to \infty$. This together with 
Lemma $\ref{lem.prophi}$ imply that $\left\|x_{n+1}-\bar{y}_n\right\| \to 0$. Therefore, $\bar{y}_n\to p$ and $||x_n - \bar{y}_n|| 
\to 0$. Further, by the definition of $i_n$, we have 
$\left\|x_n-y_n^{i}\right\| \leq ||x_n - \bar{y}_n|| \to 0$ as $n\to \infty$ for all $i=1,2,\ldots,N$, hence, from $x_n\to p$, we obtain 
\begin{equation}\label{eq:9*}
\lim\limits_{n\to \infty}y_n^i=p, \quad i=1,2,\ldots,N.
\end{equation}
Taking into account the relation $y_n^i=J^{-1}\left(\alpha_n Jx_0 +(1-\alpha_n)JT^n_i x_n\right)$, we obtain
\begin{equation}\label{eq:10*}
\left\|Jy_n^i-JT^n_i x_n\right\|=\alpha_n\left\|Jx_0-JT^n_i x_n\right\|.
\end{equation}
Observing that $\left\{x_n\right\}$ is bounded, $T_i$ is uniformly $L_i$ - Lipschitz continuous, and the 
solution set $F$ is not empty, we have $||Jx_0 - JT_i^nx_n|| \leq ||Jx_0||+||JT_i^nx_n|| = ||x_0|| + ||T_i^n x_n|| \leq ||x_0|| +||T_i^nx_n 
- T_i^n\xi|| + ||\xi|| \leq ||x_0|| + L_i||x_n - \xi|| +|\xi||$, where $\xi \in F$ is an arbitrary fixed element. The last inequality proves the 
boundedness of the sequence  $\left\{\left\|Jx_0-JT^n_i x_n\right\|\right\}$ . Using $\lim\limits_{n\to \infty}\alpha_n=0$ from $(\ref 
{eq:10*})$, we find
$$\lim\limits_{n\to \infty}\left\|Jy_n^i-JT^n_i x_n\right\|=0.$$
Since $J^{-1}:E^*\to E$ is uniformly continuous on each bounded subset of $E^*$, the last relation 
implies $\lim\limits_{n\to \infty}\left\|y_n^i-T^n_i x_n\right\|=0$. Hence, from $(\ref{eq:9*})$ we obtain
$$\lim\limits_{n\to \infty}T^n_i x_n=p, \quad i =1,\ldots, N.$$
Finally,  a similar argument as in Step 5 of Theorem $\ref{theo.converges}$ leads to the conclusion that $p\in F$ and $p=x^\dagger=\Pi_F x_0$. The proof of
Theorem $\ref{theo.extend1}$ is complete.
\end{proof}
\begin{remark}
Theorem $\ref{theo.extend1}$ is an extended version of Theorem 3.1 in \cite{CQK2009} and Corollary 2.5 in \cite{CY2013}
for a family of \textcolor[rgb]{1.0,0.0,0.0}{asympotically quasi-$\phi$-nonexpansive} mappings. It also 
simplifies some previous results of Chang and Yan (Theorem 2.1 \cite{CY2013}) and Cho, Qin, and Kang (Theorem 3.5 \cite{CQK2009}).
In the case $N=1,$  our method modifies the algorithm of Kim and Takahashi \cite{KT2010}.
\end{remark}
In the next theorem, we show that for quasi $\phi$ - nonexpansive mappings $\left\{T_i\right\}_{i=1}^N$, the assumptions on their uniform Lipschitz 
continuity, as well as the boundedness of  the set of common fixed points $F=\bigcap^N_{i=1} F(T_i)$ are redundant.
\begin{theorem}\label{theo.extend2}
Let $E$ be a real uniformly smooth and uniformly convex Banach space, $C$ be a nonempty closed convex subset of $E$, and 
$\left\{T_i\right\}_{i=1}^N:C\to C$ be a finite family of closed and quasi $\phi$ - nonexpansive mappings. Suppose that 
$F=\bigcap^N_{i=1} F(T_i) \ne \O$. Let $\left\{x_n\right\}$ be the sequence generated by
\begin{eqnarray*}
&&x_0\in C, C_0:=C,\\
&&y_n^i=J^{-1}\left(\alpha_n Jx_n +(1-\alpha_n)JT_i x_n\right),i=1,2,\ldots,N,\\
&&i_n=\arg\max_{1\le i\le N}\left\{\left\|y_n^i-x_n \right\|\right\}, \bar{y}_n:=y_n^{i_n},\\
&&C_{n+1}:=\left\{v\in C_n:\phi(v,\bar{y}_n)\le \phi(v,x_n)\right\},\\
&&x_{n+1}=\Pi_{C_{n+1}}x_0, n\ge 0,
\end{eqnarray*}
where $\left\{\alpha_n\right\}$ is a sequence in $[0,1]$ such that $\lim\limits_{n\to\infty}\alpha_n =0$.
Then $\left\{x_n\right\}$ converges strongly to $x^{\dagger}:=\Pi_F x_0$.
\end{theorem}
\begin{proof}
Since $\left\{T_i\right\}_{i=1}^N:C\to C$ are quasi $\phi$ - nonexpansive mappings, for each $i=1,\ldots, N,$ we have
$$\phi(p,T_ix)\le \phi (p,x), \forall p\in F(T_i), x\in C.$$
This implies that $\left\{T_i\right\}_{i=1}^N$ are \textcolor[rgb]{1.0,0.0,0.0}{asymptotically quasi $\phi$-nonexpansive} mappings 
with $k_n = 1,  n\geq 1$. \textcolor[rgb]{1.0,0.0,0.0}{Putting $\epsilon_n=0$} and arguing similarly as in the 
proof of Theorem $\ref{theo.converges}$, we get $F\subset C_n$. Using Lemma $\ref{lem.proProjec}$ 
and the fact that $x_n=\Pi_{C_n}x_0$ 
, we have $\phi(x_n,x_0)\le \phi(p, x_0)$ for each $p\in F$. Hence, the set $\left\{\phi(x_n,x_0)\right\}$ is bounded. This together with 
inequality $(\ref{eq:proPhi})$ implies that $\left\{x_n\right\}$ is bounded. Repeating the proof of  the relations 
$(\ref{eq:4}),(\ref{eq:9})$, we obtain 
\begin{eqnarray}
&&\lim\limits_{n\to \infty} x_n=p,\label{eq:14*}\\ 
&&\lim\limits_{n\to \infty} y^i_n=p, \, i=1,2,\ldots,N. \label{eq:14**}
\end{eqnarray}
From the equality $y_n^i=J^{-1}\left(\alpha_n Jx_n +(1-\alpha_n)JT_i x_n\right)$ we have
$$\left\|Jy_n^i-JT_i x_n\right\|=\alpha_n\left\|Jx_n-JT_i x_n\right\|.$$
Observing that $\left\{x_n\right\}\subset C$ is bounded, from the definition of quasi $\phi$ - nonexpansive mapping $T_i$, we get 
$\phi(p,T_i x_n)\le\phi(p,x_n)$ for each $p\in F$. Estimate $(\ref{eq:proPhi})$ ensures that $\left\{T_i x_n\right\}$ is 
bounded for each $i=1,\ldots,N$. Therefore, $\left\|Jx_n-JT_i x_n\right\|\le \left\|x_n\right\|+\left\|T_i x_n\right\|$. The last inequality 
implies that the sequence $\left\{\left\|Jx_n-JT_i x_n\right\|\right\}$ is bounded. Using $\lim\limits_{n\to \infty}\alpha_n=0$ we obtain
\begin{equation}\label{eq:15}
\lim\limits_{n\to \infty}\left\|Jy_n^i-JT_i x_n\right\|=0.
\end{equation}
From $(\ref{eq:14**}),(\ref{eq:15})$, by the same way as in the proof of $(\ref{eq:12})$, we get 
\begin{equation}\label{eq:16}
\lim\limits_{n\to \infty}T_i x_n=p,\,i=1,2,\ldots,N.
\end{equation}
By $(\ref{eq:14*}),(\ref{eq:16})$ and the closedness of $T_i$, we obtain $p\in F=\bigcap^N_{i=1} F(T_i)$. Finally, arguing as in Step 5 of the proof of 
Theorem $\ref{theo.converges}$, we can show that $p= x^\dagger.$ Thus, the proof of Theorem $\ref{theo.extend2}$ is complete.
\end{proof}
By the same method we can prove the following result.
\begin{theorem}\label{theo.extend3}
Let $E$ be a real uniformly smooth and uniformly convex Banach space, $C$ be a nonempty closed convex subset of $E$, and 
$\left\{T_i\right\}_{i=1}^N:C\to C$ be a finite family of closed and quasi $\phi$ - nonexpansive mappings. Suppose that 
$F=\bigcap^N_{i=1} F(T_i) \ne \O$. Let $\left\{x_n\right\}$ be the sequence generated by
\begin{eqnarray*}
&&x_0\in C, C_0:=C,\\
&&y_n^i=J^{-1}\left(\alpha_n Jx_0 +(1-\alpha_n)JT_i x_n\right),i=1,2,\ldots,N,\\
&&i_n=\arg\max_{1\le i\le N}\left\{\left\|y_n^i-x_n \right\|\right\}, \bar{y}_n:=y_n^{i_n},\\
&&C_{n+1}:=\left\{v\in C_n:\phi(v,\bar{y}_n)\le \alpha_n\phi(v,x_0)+(1-\alpha_n)\phi(v,x_n)\right\},\\
&&x_{n+1}=\Pi_{C_{n+1}}x_0, n\ge 0,
\end{eqnarray*}
where $\left\{\alpha_n\right\}$ is a sequence in $[0,1]$ such that $\lim\limits_{n\to\infty}\alpha_n =0$.
Then $\left\{x_n\right\}$ converges strongly to $x^{\dagger}:=\Pi_F x_0$.
\end{theorem}
\begin{remark}
Theorem $\ref{theo.extend2}$ modifies Theorem 3.1 \cite{SWX2009}, Theorem 3.1 \cite{ZG2009} and the algorithm in Theorem 3.2 \cite{KT2009}. On the other hand,
the method in Theorem $\ref{theo.extend3}$ simplifies the corresponding one in Theorem 3.3 \cite{SWX2009}. It generalizes  and improves 
Theorem 3.2 \cite{SLZ2011}, Theorem 3.3 \cite{CKW2010}, and Theorem 3.1 in \cite{QCKZ2009}.
\end{remark}
The following result can be obtained from Theorem $\ref{theo.extend2}$ immediately.
\begin{corollary} \label{coro.RNM}
Let $E$ be a real uniformly smooth and uniformly convex Banach space, and $C$ a nonempty closed convex subset of $E$. Let 
$\left\{T_i\right\}_{i=1}^N:C\to C$ be a finite family of closed relatively nonexpansive mappings. Suppose that 
$F=\bigcap^N_{i=1} F(T_i) \ne \O$. Let $\left\{x_n\right\}$ be the sequence generated by
\begin{eqnarray*}
&&x_0\in C, C_0=C,\\
&&y_n^i=J^{-1}\left(\alpha_n Jx_n +(1-\alpha_n)JT_i x_n\right),i=1,2,\ldots,N,\\
&&i_n=\arg\max_{1\le i\le N}\left\{\left\|y_n^i-x_n \right\|\right\}, \bar{y}_n:=y_n^{i_n},\\
&&C_{n+1}=\left\{v\in C_n:\phi(v,\bar{y}_n)\le \phi(v,x_n)\right\},\\
&&x_{n+1}=\Pi_{C_{n+1}}x_0, n\ge 0,
\end{eqnarray*}
where $\left\{\alpha_n\right\}$ is a sequence in $[0,1]$ such that $\lim\limits_{n\to\infty}\alpha_n =0$.
Then $\left\{x_n\right\}$ converges strongly to $x^{\dagger}:=\Pi_F x_0$.
\end{corollary}
\begin{corollary}\label{coro.MMM}
Let $E$ be a real uniformly smooth and uniformly convex Banach space. Let $\left\{A_i\right\}_{i=1}^N:E\to E^*$ be a finite family of 
maximal monotone mappings with $D(A_i)=E$ for all $i=1,\ldots,N$. Suppose that the solution set $S$ of the system of operator 
equations $A_i (x)=0, i=1,\ldots,N$ is nonempty. Let $\left\{x_n\right\}$ be the sequence generated by
\begin{eqnarray*}
&&x_0\in E, C_0=E,\\
&&y_n^i=J^{-1}\left(\alpha_n Jx_n +(1-\alpha_n)J(J+r_i A_i)^{-1}Jx_n\right),i=1,2,\ldots,N,\\
&&i_n=\arg\max_{1\le i\le N}\left\{\left\|y_n^i-x_n \right\|\right\},\bar{y}_n:=y_n^{i_n},\\
&&C_{n+1}=\left\{v\in C_n:\phi(v,\bar{y}_n)\le \phi(v,x_n)\right\},\\
&&x_{n+1}=\Pi_{C_{n+1}}x_0, n\ge 0,
\end{eqnarray*}
where $\left\{r_i\right\}_{i=1}^N$ are given positive numbers and $\left\{\alpha_n\right\}$ is a sequence in $[0,1]$ such that 
$\lim\limits_{n\to\infty}\alpha_n =0$.
Then $\left\{x_n\right\}$ converges strongly to $x^{\dagger}:=\Pi_S x_0$.
\end{corollary}
\begin{proof}
Let $C=D(A_i) = E$ and $T_i=(J+r_iA_i)^{-1}J:C\to C$. By Lemmas $\ref{Closed-of-Jr}$ and $\ref{lem.MNOper},$ 
the mappings $T_i,  i = 1, \ldots, N,$ are closed and quasi $\phi$-nonexpansive. Moreover,   
$F=\bigcap_{i=1}^N F(T_i)=\bigcap_{i=1}^N A_i^{-1}(0)=S\ne \O$. Thus, Theorem $\ref{theo.extend2}$ ensures the conclusion 
of  Corollary  3.10.
\end{proof}
\subsection{\bf Sequential hybrid methods}\label{subsec:2.2}
Now, we consider a sequential method for finding a common fixed point of a finite family of \textcolor[rgb]{1.0,0.0,0.0}{asymptotically quasi $\phi$-nonexpansive} mappings.
\begin{theorem}\label{T.converges.se2}
Let $C$ be a nonempty closed convex subset of  a real uniformly smooth and uniformly convex Banach space $E$, and 
$\left\{T_i\right\}_{i=1}^N:C\to C$ be a finite family of \textcolor[rgb]{1.0,0.0,0.0}{asymptotically quasi $\phi$-nonexpansive} mappings with 
$\left\{k_n\right\} \subset [1,+\infty), k_n\to 1$. Suppose $\left\{T_i\right\}_{i=1}^N$ are uniformly $L$ - Lipschitz continuous and 
the set $F=\bigcap^N_{i=1} F(T_i)$ is unempty and bounded in $C$, i.e., \textcolor[rgb]{1.0,0.0,0.0}{$F \subset \Omega := \{u \in C: ||u|| \leq \omega \}$ for some positive $\omega$}. Let $\left\{x_n\right\}$ be the sequence generated by
\begin{eqnarray*}
&&x_1\in C_1=Q_1:=C,\\
&&y_n=J^{-1}\left(\alpha_n Jx_1 +(1-\alpha_n)JT^{p_n}_{j_n} x_n\right),\\
&&C_{n}=\left\{v\in C:\phi(v,y_n)\le \alpha_n\phi(v,x_1)+(1-\alpha_n)\phi(v,x_n)+\epsilon_n\right\},\\
&&Q_n=\left\{v\in Q_{n-1}:\left\langle Jx_1-Jx_n;x_n-v\right\rangle \ge 0\right\},\\
&&x_{n+1}=\Pi_{C_{n}\bigcap Q_n}x_1, n\ge 1,
\end{eqnarray*}
where $n=(p_n -1)N+j_n, j_n \in \left\{1,2,\ldots,N\right\}$, $ p_n \in \left\{1,2,\ldots\right\}$, \textcolor[rgb]{1.0,0.0,0.0}{$ \epsilon_n=(k_{p_n}-1)(\omega+||x_n||)^2$} and $\left\{\alpha_n\right\}$ is a sequence in $[0,1]$ such that $\lim\limits_{n\to\infty}\alpha_n =0$. Then the sequence $\left\{x_n\right\}$ converges strongly to $x^{\dagger}:=\Pi_F x_1$.
\end{theorem}
For the proof of Theorem $\ref{T.converges.se2}$ we need the following result.
\begin{lemma}\label{lem.help}
Assume that all conditions of Theorem $\ref{T.converges.se2}$ holds. Moreover, 
$$\lim_{n\to \infty}\left\|x_n -T^{p_n}_{j_n} x_n\right\|=0,\,\lim_{n\to \infty}\left\|x_n -x_{n+l}\right\|=0$$
 for all $l\in \left\{1,2,\ldots,N\right\}$. Then 
$$\lim_{n\to \infty}\left\|x_n -T_l x_n\right\|=0, \quad l=1,\ldots,N.$$
\end{lemma}
\begin{proof}
For each $n>N$, we have $n=(p_n-1)N+j_n$. Hence $n-N=((p_n-1)-1)N+j_n =(p_{n-N}-1)N+j_{n-N}$. So 
$$p_n-1=p_{n-N}, j_n=j_{n-N}.$$
We have 
\begin{eqnarray*}
\left\|x_n -T_{j_n} x_n\right\|&\le&  \left\|x_n -T^{p_n}_{j_n} x_n\right\|+\left\|T^{p_n}_{j_n} x_n-T_{j_n} x_n\right\| \\ 
&\le&  \left\|x_n -T^{p_n}_{j_n} x_n\right\|+L\left\|T^{p_n-1}_{j_n} x_n-x_n\right\| \\ 
&\le&  \left\|x_n -T^{p_n}_{j_n} x_n\right\|+L\left\|T^{p_n-1}_{j_n} x_n-T^{p_n-1}_{j_{n-N}} x_{n-N}\right\| \\ 
&&+L\left\|T^{p_n-1}_{j_{n-N}} x_{n-N}-x_{n-N}\right\|+L\left\|x_{n-N}-x_n\right\| \\
&=&\left\|x_n -T^{p_n}_{j_n} x_n\right\|+L\left\|T^{p_{n-N}}_{j_{n-N}} x_{n-N}-x_{n-N}\right\|\\
&&+(L^2+L)\left\|x_{n-N}-x_n\right\|.
\end{eqnarray*}
This together with the hypotheses of Lemma $\ref{lem.help}$ implies 
$$\lim_{n\to \infty}\left\|x_n -T_{j_n} x_n\right\|=0.$$
For each $l\in\left\{1,2,\ldots,N\right\}$ we have
\begin{eqnarray*}
\left\|x_n -T_{j_{n+l}} x_n\right\|&\le&  \left\|x_n - x_{n+l}\right\|+\left\|x_{n+l}-T_{j_{n+l}} x_{n+l}\right\| +\left\|T_{j_{n+l}} x_{n+
l}-T_{j_{n+l}} x_n\right\|\\ 
&\le&\left\|x_n - x_{n+l}\right\|+\left\|x_{n+l}-T_{j_{n+l}} x_{n+l}\right\| +L\left\|x_{n+l}-x_n\right\| \\
&=&(1+L)\left\|x_n - x_{n+l}\right\|+\left\|x_{n+l}-T_{j_{n+l}} x_{n+l}\right\|.
\end{eqnarray*}
Hence, $\lim_{n\to \infty}\left\|x_n -T_{j_{n+l}} x_n\right\|=0$ for all $l\in\left\{1,2,\ldots,N\right\}$; therefore,
$$ \forall \epsilon>0,\exists n_0:\forall n\ge n_0 ~  \forall l=1,\ldots,N,||x_n-T_{j_{n+l}}x_n||<\epsilon. $$
On the other hand, for any fixed $n\ge 0$ and $i=1,\ldots,N$, we can find $l\in \left\{1,\ldots,N\right\}$, such that $i=j_{n+l}$. Thus, $||x_n-T_i x_n||\le \sup_{l\in \left\{1,\ldots,N\right\}}||x_n-T_{j_{n+l}}x_n||<\epsilon$ for all $n\ge n_0$, which means that $\lim_{n\to \infty}\left\|x_n -T_i x_n\right\|=0,i=1,\ldots, N.$
 The proof of Lemma $\ref{lem.help}$ is complete.

\textit{Proof of Theorem $\ref{T.converges.se2}$}.
The proof will be divided into five steps.\\
\textit{Step 1.} The sets $F, C_n, Q_n$ are closed and convex for all $n\ge 1$. \\
Indeed, from the uniform $L$-Lipschitz continuty of $T_i$, we see that $T_i$ is closed. By Lemma $\ref{lem.proCloseConvex}$, 
$F(T_i)$ is closed and convex subset of $C$ for all $i=1,\ldots, N$. Hence, $F=\bigcap^N_{i=1} F(T_i)$ is closed and convex. 
Further, $C_n$ and $Q_n$ are closed for all $n\ge 1$ by the definition. From the inequality $\phi(v,y_n)\le \alpha_n\phi(v,x_1)+(1-\alpha_n) \phi(v,x_n)+\epsilon_n$, we obtain
$$2\left\langle v,Jx_n\right\rangle+2\alpha_n\left\langle v,Jx_1-Jy_n-Jx_n\right\rangle\le \alpha_n\left\|x_1 \right\|^2+(1-\alpha_n)\left\|x_n\right\|^2-\left\|y_n\right\|^2+\epsilon_n,$$
which implies the convexity of  $C_n$ for all $n\ge 1$. Further, $Q_1=C$ is convex. If $Q_n$ is convex for some $n\ge 1$, 
then $Q_{n+1}$  is also 
convex by the definition. So, $Q_n$ is convex for all $n\ge 1$.\\
\textit{Step 2.} $F\subset C_n \bigcap Q_n$ for all $n\ge 1$. \\
For each $u\in F$, we have
\begin{eqnarray*}
\phi(u,y_n)&=&\left\|u\right\|^2-2\left\langle u,Jy_n\right\rangle+\left\|y_n\right\|^2\\ 
&=& \left\|u\right\|^2-2\alpha_n\left\langle u,Jx_1\right\rangle -2(1-\alpha_n)\left\langle u,JT_{j_n}^{p_n} x_n\right\rangle\\
&&+\left\|\alpha_n Jx_1 +(1-\alpha_n)JT_{j_n}^{p_n} x_n \right\|^2\\
&\le& \left\|u\right\|^2-2\alpha_n\left\langle u,Jx_1\right\rangle -2(1-\alpha_n)\left\langle u,JT_{j_n}^{p_n} x_n\right\rangle\\
&&+\alpha_n\left\|x_1\right\|^2+(1-\alpha_n)\left\|T_{j_n}^{p_n} x_n\right\|^2 \\
&=&\alpha_n\phi(u,x_1)+(1-\alpha_n)\phi(u,T_{j_n}^{p_n} x_n)\\
&\le& \alpha_n\phi(u,x_1)+k_{p_n}(1-\alpha_n)\phi(u,x_n)\\
&\le& \alpha_n\phi(u,x_1)+(1-\alpha_n)\phi(u,x_n)+(k_{p_n}-1)(1-\alpha_n)\phi(u,x_n) \\
&\le& \alpha_n\phi(u,x_1)+(1-\alpha_n)\phi(u,x_n)+(k_{p_n}-1)(\omega+||x_n||^2)\\
&=& \alpha_n\phi(u,x_1)+(1-\alpha_n)\phi(u,x_n)+\epsilon_n.
\end{eqnarray*}
This implies that $u\in C_{n}$. Hence $F\subset C_{n}$ for all $n\ge 1$. We also have $F\subset Q_1=C$. Suppose that $F\subset 
Q_n$ for some $n\ge 1$. From $x_{n+1}=\Pi_{C_{n}\bigcap Q_n}x_1$ and Lemma $\ref{lem.proProjec}$, it follows that 
$\left\langle Jx_1 - Jx_{n+1},x_{n+1}-z\right\rangle \ge 0$ for all $z\in C_n \bigcap Q_n$. Since $F\subset C_n \bigcap Q_n$, we have
 $$\left\langle Jx_1 - Jx_{n+1},x_{n+1}-z\right\rangle \ge 0$$ 
for all $z\in F$. Hence, from the definition of $Q_{n+1}$, we obtain $F\subset Q_{n+1}$. By the induction, $F\subset Q_n$ for all 
$n\ge 1$.\\
\textit{Step 3.} $\lim_{n\to\infty}\left\| x_n -T_l x_n\right\|=0$ for all $l=1,\ldots,N$. \\
Since $x_n=\Pi_{Q_n}x_1, F\subset Q_n$, by Lemma $\ref{lem.proProjec}$, we have $ \phi(x_n,x_1)\le \phi(p,x_1)-\phi(x_n,p)\le 
\phi(p,x_1)$ for each $p\in F$. Hence, the 
sequence $\left\{\phi(x_n,x_1)\right\}$ and $\left\{x_n\right\}$ are bounded. Moreover, from $x_{n+1}=\Pi_{C_{n}\bigcap 
Q_n}x_1\in Q_n$, $x_n=\Pi_{Q_n}x_1$ and Lemma $\ref{lem.proProjec}$, it follows that $\phi(x_n,x_1)\le \phi(x_{n+1},x_1)$. Thus, 
the sequence $ \left\{\phi(x_n,x_1)\right\} $ is nondecreasing and the limit of the sequence $\left\{\phi(x_n,x_1)\right\}$ exists. 
This together with $\phi(x_{n+1},x_n)\le \phi(x_n,x_1)+\phi(x_{n+1},x_1)$, implies that 
\begin{equation}\label{eq:27}
\lim_{n\to \infty} \phi(x_{n+1},x_n)=0.
\end{equation}
Since $\left\{x_n\right\} $ is bounded, there exists $M>0$ such that $\left\|x_n\right\|\le M$ for all $n\ge 1$. Using the boundedness of  $F$ 
and estimate $(\ref{eq:proPhi})$, we get 
\begin{equation}\label{eq:28}
\textcolor[rgb]{1.0,0.0,0.0}{\epsilon_n=(k_{p_n}-1)\left(\omega+\left\|x_n\right\|\right)^2\le(k_{p_n}-1)\left(\omega+M\right)^2 \to 0 }
\, (n\to \infty).
\end{equation}
Taking into account $x_{n+1}=\Pi_{C_{n}\bigcap Q_n}x_1 \in C_n$, and using the relations $(\ref{eq:27}),(\ref{eq:28})$, 
and $\lim_{n\to \infty} \alpha_n =0$, from the definition of $C_n$  we find
$$\phi(x_{n+1},y_n)\le \alpha_n\phi(x_{n+1},x_1)+(1-\alpha_n)\phi(x_{n+1},x_n)+\epsilon_n \to 0  \, (n\to \infty).$$
Lemma $\ref{lem.prophi}$ gives
$$\lim_{n\to \infty}\left\|x_{n+1}-y_n\right\|=\lim_{n\to \infty}\left\|x_{n+1}-x_n\right\|=\lim_{n\to \infty}\left\|x_{n}-y_n\right\|=0.$$
and
\begin{equation}\label{eq:30}
\lim_{n\to \infty}\left\|x_{n+l}-x_n\right\|=0
\end{equation}
for all $l\in \left\{1,2,\ldots,N\right\}$.
Note that from $y_n=J^{-1}\left(\alpha_n Jx_1 +(1-\alpha_n)JT^{p_n}_{j_n} x_n\right)$, we have
\begin{equation}\label{eq:31}
\left\|Jy_n -JT^{p_n}_{j_n} x_n\right\|=\alpha_n\left\|Jx_1 -JT^{p_n}_{j_n} x_n\right\|.
\end{equation}
Observing that $\left\{x_n\right\}$ is bounded, $T_{j_n}$ is uniformly $L$ - Lipschitz continuous and the solution set $F$ is not empty
, we have $||Jx_1 - JT^{p_n}_{j_n} x_n|| \leq ||Jx_1||+||JT^{p_n}_{j_n} x_n|| = ||x_1|| + ||T^{p_n}_{j_n} x_n|| \leq ||x_1|| 
+||T^{p_n}_{j_n} x_n - T^{p_n}_{j_n}\xi|| + ||\xi|| \leq ||x_1|| + L||x_n - \xi|| +|\xi||$, where $\xi \in F$ is an arbitrary fixed element. 
The last inequality proves the boundedness of the sequence  $\left\{\left\|Jx_1 -JT^{p_n}_{j_n} x_n\right\|\right\}$. Using 
$\lim\limits_{n\to \infty}\alpha_n=0$, from $(\ref {eq:31})$, we find
$$\lim\limits_{n\to \infty}\left\|Jy_n -JT^{p_n}_{j_n} x_n\right\|=0.$$
Since $J^{-1}:E^* \to E$ is uniformly continuous on each bounded set, we get
$$\lim_{n\to \infty}\left\|y_n -T^{p_n}_{j_n} x_n\right\|=0.$$
This together with $\lim_{n\to \infty}\left\|x_{n}-y_n\right\|=0$ implies that
\begin{equation}\label{eq:33}
\lim_{n\to \infty}\left\|x_n -T^{p_n}_{j_n} x_n\right\|=0.
\end{equation}
From $(\ref{eq:30}),(\ref{eq:33})$ and Lemma $\ref{lem.help}$, we obtain 
\begin{equation}\label{eq:34}
\lim_{n\to \infty}\left\|x_n -T_l x_n\right\|=0
\end{equation}
for all $l\in \left\{1,2,\ldots,N\right\}$. \\
\textit{Step 4.} $\lim_{n\to\infty}x_n=p \in F$. \\
Indeed, note that the limit of the sequence $\left\{\phi(x_n,x_1)\right\}$ exists. By the 
construction of $Q_n$, we have $Q_m\subset Q_n$ for all $m\ge n$. Moreover, $x_n=\Pi_{Q_n} x_1$ and $x_m\in Q_m\subset 
Q_n$. These together with Lemma $\ref{lem.proProjec}$ imply that $\phi(x_m,x_n)\le \phi(x_m,x_1)-\phi(x_n,x_1)\to 0$ as  $m,n\to 
\infty$. By Lemma $\ref{lem.prophi}$, we get \\
$\lim_{m,n\to\infty}\left\| x_m - x_n\right\|=0$. Hence, $\left\{x_n\right\}$ is a Cauchy 
sequence. Since $C$ is a closed and convex subset of the Banach space $E$, the sequence $\left\{x_n\right\}$ converges strongly to 
$p\in C$. Since $T_l$ is $L$ - Lipschitz continuous mapping, it is continuous for all $l\in \left\{1,2,\ldots,N\right\}$. Hence
$$\left\|p -T_l p\right\|=\lim_{n\to \infty}\left\|x_{n} -T_l x_{n}\right\|=0,\, \forall l\in \left\{1,2,\ldots,N\right\}.$$
This implies that $p\in F$.\\
\textit{Step 5.} $p=\Pi_F x_1$. \\
From $x^\dagger:=\Pi_F x_1 \in F \subset C_n \bigcap Q_n$ and $x_{n+1}=\Pi_{C_n \bigcap Q_n} 
x_1$, we have $\phi\left(x_{n+1},x_1\right)\le \phi \left(x^\dagger,x_1\right)$. Hence
$$\phi \left(p,x_1\right)=\lim_{n\to\infty}\phi\left(x_{n},x_1\right) \le \phi \left(x^\dagger,x_1\right).$$
Therefore, $p=x^\dagger$. The proof of Theorem $\ref{T.converges.se2}$ is complete.
\end{proof}
For a finite family of closed and quasi $\phi$ - nonexpansive mappings, the assumption on the boundedness of 
$F=\bigcap^N_{i=1} F(T_i)$ is redundant.

\begin{theorem}\label{T.converges.se3}
Let $E$ be a real uniformly smooth and uniformly convex Banach space, and $C$ a nonempty closed convex subset of $E$. Let 
$\left\{T_i\right\}_{i=1}^N:C\to C$ be a finite family of closed and quasi $\phi$ - nonexpansive mappings. Suppose 
$\left\{T_i\right\}_{i=1}^N$ are $L$ - Lipschitz continuous and $F=\bigcap^N_{i=1} F(T_i) \ne \O$. Let $\left\{x_n\right\}$ be the 
sequence generated by
\begin{eqnarray*}
&&x_1\in C_1=Q_1:=C,\\
&&y_n=J^{-1}\left(\alpha_n Jx_1 +(1-\alpha_n)JT_{j_n} x_n\right),\\
&&C_{n}=\left\{v\in C:\phi(v,y_n)\le \alpha_n\phi(v,x_1)+(1-\alpha_n)\phi(v,x_n)\right\},\\
&&Q_n=\left\{v\in Q_{n-1}:\left\langle Jx_1-Jx_n;x_n-v\right\rangle \ge 0\right\},\\
&&x_{n+1}=\Pi_{C_{n}\bigcap Q_n}x_1, n\ge 1,
\end{eqnarray*}
where $n=(p_n -1)N+j_n, j_n \in \left\{1,2,\ldots,N\right\}$ and $\left\{\alpha_n\right\}$ is a sequence in $[0,1]$ such that 
$\lim\limits_{n\to\infty}\alpha_n =0$. Then the sequence $\left\{x_n\right\}$ converges strongly to $x^{\dagger}:=\Pi_F x_1$.
\end{theorem}
\begin{proof}
By our assumption, $\left\{T_i\right\}_{i=1}^N$ is a finite family of closed and \textcolor[rgb]{1.0,0.0,0.0}{asymptotically quasi $\phi$-nonexpansive} 
mappings with 
$k_n=1$ for all $n \ge 0$. Putting $\epsilon_n=0$ and arguing similarly as in the proofs of Theorem $\ref{T.converges.se2}$ and 
Lemma $\ref{lem.help}$, we obtain $\lim_{n\to \infty}\left\|x_n -T_{j_n} x_n\right\|=0$ and $\lim_{n\to \infty}\left\|x_n -T_{l} 
x_n\right\|=0$ for all $l\in \left\{1,2,\ldots,N\right\}$. Now repeating Steps 4 and 5 of  the proof of Theorem $\ref{T.converges.se2}$, 
we come to the conclusion of Theorem $\ref{T.converges.se3}$.
\end{proof}
\begin{remark}
One can  establish the convergence of a monotone hybrid method as in Theorem $\ref{T.converges.se2}$, which modifies Liu's algorithm \cite{L2011}.
\end{remark}
\begin{corollary}\label{coro.se.RNM4} 
Let $C$ be a nonempty closed convex subset of a real uniformly smooth and uniformly convex Banach space $E$. Let 
$\left\{T_i\right\}_{i=1}^N:C\to C$ be a finite family of closed relatively nonexpansive mappings. Suppose 
$\left\{T_i\right\}_{i=1}^N$ are $L$ - Lipschitz continuous and $F=\bigcap^N_{i=1} F(T_i) \ne \O$. Let $\left\{x_n\right\}$ be the 
sequence generated by
\begin{eqnarray*}
&&x_1\in C_1=Q_1:=C,\\
&&y_n=J^{-1}\left(\alpha_n Jx_1 +(1-\alpha_n)JT_{j_n} x_n\right),\\
&&C_{n}=\left\{v\in C:\phi(v,y_n)\le \alpha_n\phi(v,x_1)+(1-\alpha_n)\phi(v,x_n)\right\},\\
&&Q_n=\left\{v\in Q_{n-1}:\left\langle Jx_1-Jx_n;x_n-v\right\rangle \ge 0\right\},\\
&&x_{n+1}=\Pi_{C_{n}\bigcap Q_n}x_1, n\ge 1,
\end{eqnarray*}
where $n=(p_n -1)N+j_n, j_n \in \left\{1,2,\ldots,N\right\}$ and $\left\{\alpha_n\right\}$ is a sequence in $[0,1]$ such that 
$\lim\limits_{n\to\infty}\alpha_n =0$. Then the sequence $\left\{x_n\right\}$ converges strongly to $x^{\dagger}:=\Pi_F x_1$.
\end{corollary}
\begin{corollary}\label{coro.MMM}
Let $E$ be a real uniformly smooth and smooth convex Banach space. Let $\left\{A_i\right\}_{i=1}^N:E\to E^*$ be a finite family of 
maximal monotone mappings with $D(A_i)=E$ for all $i=1,\ldots,N$. Suppose that the solution set $S$ of the system of operator 
equations $A_i (x)=0, i=1,\ldots,N$ is nonempty. Let $\left\{x_n\right\}$ be the sequence generated by
\begin{eqnarray*}
&&x_1\in E, C_1=E,\\
&&y_n=J^{-1}\left(\alpha_n Jx_1 +(1-\alpha_n)J(J+r_{j_n} A_{j_n})^{-1}Jx_n\right),i=1,2,\ldots,N,\\
&&C_{n}=\left\{v\in C:\phi(v,y_n)\le \alpha_n\phi(v,x_1)+(1-\alpha_n)\phi(v,x_n)\right\},\\
&&Q_n=\left\{v\in Q_{n-1}:\left\langle Jx_1-Jx_n;x_n-v\right\rangle \ge 0\right\},\\
&&x_{n+1}=\Pi_{C_{n}\bigcap Q_n}x_1, n\ge 1,
\end{eqnarray*}
where $\left\{r_i\right\}_{i=1}^N$ are given positive numbers and $\left\{\alpha_n\right\}$ is a sequence in $[0,1]$ such that 
$\lim\limits_{n\to\infty}\alpha_n =0$.
Then $\left\{x_n\right\}$ converges strongly to $x^{\dagger}:=\Pi_S x_1$.
\end{corollary}
\textcolor[rgb]{1.0,0.0,0.0}{We end this paper by considering a numerical example.
Suppose we are given two sequences of positive numbers $0 < t_1 < \ldots < t_N  < 1$ and $  s_i \in (1,  \frac{1}{1- t_i}];    i =1, \ldots, 
N.$ An example of such $\{s_i\}_{i=1}^N$  are  $s_i = \sum_{k=0}^{m_i} t_i^k,$ where the integers $m_i \geq 1$ for all $i 
=1,\ldots, N$.\\
 Let $E =  R^1$ be a Hilbert space with the standart inner product $\left\langle x,y \right\rangle := xy$ 
and the norm $||x|| := |x|$ 
for all $x, y \in E$.  In this case the normalized dual mapping $J=I$ and the Lyapunov functional $\phi (x,y) = |x-y|^2.$   We define the 
mappings $T_i: C \to C,    i =1,\ldots, N,$ where $C := [0,1],$ as follows:\\
$ T_i(x) = 0, $ for $ x \in [0, t_i],$ and $T_i(x) = s_i(x-t_i),$ if $x \in [t_i, 1]$.}\\
\textcolor[rgb]{1.0,0.0,0.0}{It is easy to verify that $F(T_i) = \{0\},  \phi(T_i(x), 0)= |T_i(x)|^2 \leq |x|^2 = \phi(x,0) $ for every $x\in C$ and $|T_i(1) - T_i(t_i)| = s_i 
(1-t_i) > |1-t_i|.$
Hence, the mappings $T_i$ are quasi $\phi$-nonexpansive but not nonexpansive.}\\
\textcolor[rgb]{1.0,0.0,0.0}{According to Theorem 3, the iteration sequence $\{x_n\}$ generated by
\begin{eqnarray*}
&&x_0\in C, C_0:=C,\\
&&y_n^i=\alpha_n x_n +(1-\alpha_n)T_i x_n,   i=1,2,\ldots,N,\\
&&i_n=\arg\max_{1\le i\le N}\left\{|y_n^i-x_n |\right\}, \bar{y}_n:=y_n^{i_n},\\
&&C_{n+1}:=\left\{v\in C_n: |v - \bar{y}_n| \le |v - x_n| \right \},\\
&&x_{n+1}=\Pi_{C_{n+1}}x_0, n\ge 0,
\end{eqnarray*}
strongly converges to $x^\dagger := 0,$ provided the sequence $\{\alpha_n\}$ is chosen such that $\alpha_n \in [0, 1]$ and $\alpha_n \to 0$
 as $n \to \infty.$\\
Starting from $C_0 = C = [0, 1]$ we have 
\begin{equation}\label{eq:35}
C_1 = \{v \in C_0 : 2(\bar{y}_0-x_0) (\frac{x_0 + \bar{y}_0}{2}-v) \leq 0 \}.
\end{equation} }
\textcolor[rgb]{1.0,0.0,0.0}{Due to the proof of Theorem 3,  $F = \{0\} \subset C_1,$ hence $(\bar{y}_0 - x_0) ( \frac{x_0 + \bar{y}_0}{2}) \leq 0.$ Thus, 
$\bar{y}_0 \leq x_0.$  If $\bar{y}_0=x_0$ then from the definition of $i_0$, we find $y_0^i=x_0$ for all $i=1,...,N$. Moreover, since  
$y_0^i=\alpha_0 x_0 +(1-\alpha_0)T_i x_0$, we get $x_0=\alpha_0 x_0 +(1-\alpha_0)T_i x_0,    i=1,\ldots, N,$ hence, $x_0$ is a desired 
common fixed point and the algorithm finishes at step $n=0$.
Now suppose that $\bar{y}_0<x_0$. Then $\left(\ref{eq:35}\right)$ implies that $C_1=[0,\frac{x_0+\bar{y}_0}{2}]$ and 
$x_1=\Pi_{C_1}x_0 = \frac{x_0 + \bar{y}_0}{2}.$\\
We assume by induction that at the $n$-th step ($n \geq 1$), either $x_{n-1}$ is a common fixed point of $T_i,  i=1,\ldots,N,$ and the 
algorithm finishes at the ($n-1$)-step,  or  $C_{n}=[0,\frac{x_{n-1}+\bar{y}_{n-1}}{2}]$ and $x_n=\Pi_{C_n}x_0 = \frac{x_{n-1} + 
\bar{y}_{n-1}}{2}.$  By the definition of $C_{n+1}$ we have\\
$C_{n+1} =  \{v \in C_n: 2(\bar{y}_n-x_n) (\frac{x_n + \bar{y}_n}{2}-v) \leq 0 \}, $ or equivalently,
\begin{equation}\label{eq:36}
C_{n+1}= [0,\frac{x_{n-1}+\bar{y}_{n-1}}{2}]\bigcap \{v \in [0,1]: 2(\bar{y}_n-x_n) (\frac{x_n + \bar{y}_n}{2}-v) \leq 0 \}
\end{equation}}
\textcolor[rgb]{1.0,0.0,0.0}{Since $F = \{0\} \subset C_{n+1},$ we find that $(\bar{y}_n - x_n) ( \frac{x_n + \bar{y}_n}{2}) \leq 0,$ hence $\bar{y}_n \leq x_n.$ If 
$\bar{y}_n=x_n$ then by the definition of $i_n$, we get $y_n^i=x_n$ for all $i=1,...,N$. On the other hand, $y_n^i=\alpha_n x_n 
+(1-\alpha_n)T_i x_n$, hence, $x_n=\alpha_n x_n +(1-\alpha_n)T_i x_n$. Thus, $x_n$ is a common fixed point of the family 
$\{T_i\}_{i-1}^N$ and the algorithm finishes at the $n$-th step. In the remaining case  $\bar{y}_n<x_n,$  relation $(\ref{eq:36})$ 
gives
\begin{equation}\label{eq:37}
C_{n+1}=[0,\frac{x_{n-1}+\bar{y}_{n-1}}{2}]\bigcap [0,\frac{x_{n}+\bar{y}_{n}}{2}].
\end{equation}}
\textcolor[rgb]{1.0,0.0,0.0}{Noting that $\frac{x_{n}+\bar{y}_{n}}{2}< x_n=\frac{x_{n-1}+\bar{y}_{n-1}}{2},$ and using  $(\ref{eq:37})$ we come to the conclusion that
$C_{n+1}=[0,\frac{x_{n}+\bar{y}_{n}}{2}],$ and $x_{n+1}=\Pi_{C_{n+1}}x_0 = \frac{x_n + \bar{y}_n}{2}.$}\\
\textcolor[rgb]{1.0,0.0,0.0}{On the other hand, applying Liu's sequential method [20], at the $n$-th iteration, we need to compute $  y_n := \alpha_n x_0 + (1-\alpha_n) T_{k_n}x_n,$  where $k_n = n({\rm mod}N) + 1.$  Observing that $0 \leq T_{k_n}x_n \leq x_n \leq 1,$ we  have if $x_n = T_{k_n}x_n$ then $x_n$ is a fixed point of $T_{k_n},$ which is also a common fixed point of the family $\{T_i\}_{i=1}^N.$ Otherwise, we get $ T_{k_n}x_n < x_n,$ which leads to the formula} 
$$x_{n+1} = \min \{x_n, \frac{\alpha_n x_0^2+(1-\alpha_n)x_n^2-y_n^2}{2(\alpha_nx_0+(1-\alpha_n)x_n-y_n)}\}. $$
\textcolor[rgb]{1.0,0.0,0.0}{The numerical experiment is performed on a LINUX cluster 1350 with 8 computing nodes. Each node contains two Intel Xeon dual core 3.2 GHz, 2GBRam. All the programs are written in C.\\
For given tolerances we compare execution time of  the parallel hybrid method (PHM) and Liu's sequential method (LSM) [20]. From 
Tables 1-3, we see that within a given tolerance, the sequential method is more time consuming than the parallel one, in both parallel and sequential mode. Further, whenever the tolerance is small, the sequential method converges very slowly or practically diverges.\\
We use the following notations:}
\textcolor[rgb]{1.0,0.0,0.0}{\begin{center}
\begin{tabular}{l l}
PHM & The parallel hybrid method\\
LSM & Liu's sequential method [20]\\
$N$ & Number of  quasi $\phi$-nonexpansive mappings\\
$TOL$ & Tolerance $\|x_k - x^*\|$ \\
very slow conv. & Convergence is very slow or divergence \\
$T_p$  & Time for PHM's execution in parallel mode (2CPUs - in seconds)\\
$T_s$ & Time for PHM's execution in sequential mode (in seconds) \\
$T_L$  & Time for LSM's execution (in seconds).
\end{tabular}
\end{center}}
\textcolor[rgb]{1.0,0.0,0.0}{We perform experiments with $N = 5\times10^6, \quad t_i = \frac{i}{N+1}, \quad s_i = 1+t_i, \quad  i=1,\ldots,N.$}
\begin{table}[ht]\caption{Experiment with $\alpha_n = 1/n. $}\label{tab:1}
\medskip\begin{center}
\begin{tabular}{|c|c|c|c|}
\hline
 $\qquad TOL \qquad$ & \multicolumn{2}{c|}{PHM} &{LSM}
\\ \cline{2-4}
  & $\qquad T_p\qquad$ & $\qquad T_s\qquad$ & $\qquad T_L\qquad$\\ \hline
 $10^{-5}$ & 1.06 & 1.90 &  Very slow conv. \\
 $10^{-6}$ & 1.26 & 2.10 &  Very slow conv. \\
  $10^{-8}$ & 1.47 & 2.74 &  Very slow conv. \\ \hline
 \end{tabular}\end{center}
\end{table}
\begin{table}[ht]\caption{Experiment with $\alpha_n = \frac{1}{\log n +2}. $}\label{tab:2}
\medskip\begin{center}
\begin{tabular}{|c|c|c|c|}
\hline
 $\qquad TOL \qquad$ & \multicolumn{2}{c|}{PHM} &{LSM}
\\ \cline{2-4}
& $\qquad T_p\qquad$ & $\qquad T_s\qquad$ & $\qquad T_L\qquad$ \\ \hline
 $10^{-5}$ & 1.27 & 2.52 &  Very slow conv. \\
 $10^{-6}$ & 1.48 & 2.95 &  Very slow conv. \\
  $10^{-8}$ & 1.89 & 3.58 &  Very slow conv. \\ \hline
 \end{tabular}\end{center}
\end{table}
\begin{table}[ht]\caption{Experiment with $\alpha_n = 10^{-n}. $}\label{tab:3}
\medskip\begin{center}
\begin{tabular}{|c|c|c|c|}
\hline
 $\qquad TOL \qquad$ & \multicolumn{2}{c|}{PHM} &{LSM}
\\ \cline{2-4}
  & $\qquad T_p\qquad$ & $\qquad T_s\qquad$ & $\qquad T_L\qquad$ \\ \hline
 $10^{-5}$ & 0.84 & 1.68 &  Very slow conv. \\
 $10^{-6}$ & 1.05 & 1.90 &  Very slow conv. \\
  $10^{-8}$ & 1.26 & 2.31&   Very slow conv. \\ \hline
 \end{tabular}\end{center}
\end{table}
\textcolor[rgb]{1.0,0.0,0.0}{Within  the tolerance $TOL=10^{-4},$ for $\alpha_n = 1/n $ and  $\alpha_n = 10^{-n}$, the computing times of Liu's method  are  30.89 sec. and 26.57 sec.,  respectively.  Moreover, for $\alpha_n =1/(\log n +2),$ after 287.25 sec., Liu's method gives an approximate solution $\tilde{x}=0.327,$ which is very far from the exact solution $x^* =0.$  When  $TOL = 10^{-k}, \quad k = 5, 6, 8, $ Liu's method is practically divergent.} \\
\textcolor[rgb]{1.0,0.0,0.0}{Tables 1-3 give the execution times of the parallel hybrid method in parallel mode ($T_p$) and sequential mode ($T_s$) within the given tolerances TOL for different choices of $\alpha_n.$  The maximal speed up of the parallel hybrid method is $S_p := T_s/T_p  \approx 2.0,$ hence, the efficency of the parallel computation by using two processors is $E_p := S_p / 2 \approx 1.0.$
}
\begin{acknowledgements}
\textcolor[rgb]{1.0,0.0,0.0}{The authors are greateful to the referees for their useful comments to improve this article. We thank V. T. 
Dzung for performing computation on the LINUX cluster 1350}.  The research of the first author was partially supported by Vietnam Institute for Advanced Study in Mathematics (VIASM) and  
Vietnam National Foundation for Science and Technology Development (NAFOSTED).
\end{acknowledgements}

\end{document}